\documentclass[11pt]{article}
\usepackage{amsthm, amsmath, amssymb, amsfonts, url, booktabs, tikz, setspace, fancyhdr, bm, enumerate}
\usepackage{mathrsfs} 
\usepackage{xcolor}
\usepackage[margin = 2.3cm]{geometry}
\usepackage{hyperref, enumerate}
\usepackage[shortlabels]{enumitem}
\usepackage[babel]{microtype}
\usepackage[capitalise]{cleveref}
\usepackage{comment}
\usepackage{bbm}
\usepackage{csquotes}
\usepackage{mathabx}
\usepackage{tikz}
\usepackage{graphicx}
\usepackage{float}
\usepackage{thm-restate}
\usepackage{mathtools}

\counterwithin{figure}{section}

\newlist{Case}{enumerate}{2}
\setlist[Case, 1]{%
    label           =   {\bfseries Case \arabic*.},
    labelindent=1em ,labelwidth=1.3cm, labelsep*=1em, leftmargin =!
}
\setlist[Case, 2]{%
    label           =   {\bfseries Subcase \arabic{Casei}.\arabic*.},
    labelindent=-1em ,labelwidth=1.3cm, labelsep*=1em, leftmargin =!
}

\usepackage{todonotes}

\title{Isodiametric inequality for vector spaces}
\author{
Jiaqi Liao\thanks{State Key Laboratory of Mathematical Sciences, Academy of Mathematics and Systems Science, Chinese Academy of Sciences, Beijing 100190, China  \&  School of Mathematical Sciences, University of Chinese Academy of Sciences, Beijing 100049, China.  Emails: {\texttt 975497560@qq.com, yangy@amss.ac.cn}. Supported by National Key RD Program of China No. 2023YFA100960}
\and 
Hong Liu\thanks{Extremal Combinatorics and Probability Group (ECOPRO), Institute for Basic Science (IBS), Daejeon, South Korea. Emails: {\texttt hongliu@ibs.re.kr}. Supported by Institute for Basic Science IBS-R029-C4.}
\and
Guiying Yan\footnotemark[1]
}
\date{}

\newtheorem{theorem}{Theorem}[section]
\newtheorem{lemma}[theorem]{Lemma}
\newtheorem{proposition}[theorem]{Proposition}

\theoremstyle{definition}

\theoremstyle{remark}

\numberwithin{equation}{section}

\newcommand{\CE}{\mathcal{E}}
\newcommand{\CF}{\mathcal{F}}
\newcommand{\CG}{\mathcal{G}}

\newcommand{\CV}{\mathcal{V}}

\newcommand{\BF}{\mathbb{F}}

\newcommand{\BN}{\mathbb{N}}

\newcommand{\gs}{\geqslant}
\newcommand{\ls}{\leqslant}
\newcommand{\cb}[1]{\left\{#1\right\}}
\newcommand{\pt}[1]{\left(#1\right)}
\newcommand{\an}[1]{\langle#1\rangle}
\newcommand{\abs}[1]{\left|#1\right|}
\newcommand{\ce}[1]{\lceil#1\rceil}
\newcommand{\fl}[1]{\lfloor#1\rfloor}
\newcommand{\qbinom}[2]{\binom{#1}{#2}_q}
\newcommand{\diam}{\mathrm{diam}}
\newcommand{\supp}{\mathrm{supp}}

\begin{document}
\maketitle

\begin{abstract}
	
	A theorem of Kleitman states that a collection of binary vectors with diameter $d$ has cardinality at most that of a Hamming ball of radius $d/2$. In this paper, we give a $q$-analog of it.
	
\end{abstract} \maketitle
	
\section{Introduction}

The classical isodiametric inequality in euclidean space states that the volume of a set with given diameter is maximized by euclidean balls. Analogues of isodiametric inequalities have been established in various settings. In the discrete setting, resolving a conjecture of Erd\H{o}s, Kleitman~\cite{MR200179} in 1966 famously proved an isodiametric inequality for Hamming space, the space of binary vectors equipped with Hamming distance. We shall state his result via graphs. Given a graph $G = (V, E)$, let $\delta_G(\cdot,\cdot): V^2\rightarrow \mathbb{N}$ be the graph distance, i.e. the length of the shortest path between two vertices. For a subset $\CF\subseteq V$, its diameter is $\diam(\CF):=\max_{a,b\in \CF}\delta_G(a,b)$.
Note that $(V,\delta_G)$ is a metric space. The \emph{Hamming graph on $[n]$} has vertex set $2^{[n]}$ and two vertices $A,B\in 2^{[n]}$ are adjacent if $A\subseteq B$ and $\abs{B}-\abs{A}=1$. Kleitman's theorem is an isodiametric inequality on Hamming graphs, which reads as follows.

\begin{theorem}[Kleitman, \cite{MR200179}]\label{Kleitman}
	
	Let $n > d$  and $G$ be the Hamming graph on $[n]$. Given $\CF\subseteq V(G)$, if $\diam(\CF) \ls d$, then	
	\[\abs{\CF} \ls \left\{\begin{aligned}
		&\sum_{i = 0}^{t}\binom{n}{i}&\qquad&\text{if } d = 2t;\\
		&\sum_{i = 0}^{t}\binom{n}{i} + \binom{n - 1}{t}&\qquad&\text{if } d = 2t + 1.
	\end{aligned}\right.\]	
	
\end{theorem}

The bounds above are optimal. When $d=2t$, consider the Hamming ball of radius $t$, i.e.~$\CF=\sum_{i = 0}^{t}\binom{[n]}{i}$. When $d=2t+1$, consider $\CF=\sum_{i = 0}^{t}\binom{[n]}{i}\cup\{A\in\binom{[n]}{t+1}: 1\in A\}$.
The original proof of Kleitman is combinatorial (see also~\cite{MR465883} or \cite{MR3822342}). Recently, Huang, Klurman and Pohoata~\cite{MR4017966} gave a nice linear algebraic proof. Extensions of this theorem have been explored for the $n$-dimensional grid $[m]^n$ using Hamming distance \cite{MR1612850, MR593848}, as well as for $[m]^n$ and the $n$-dimensional torus $\mathbb{Z}_m^n$ with Manhattan distance \cite{MR1167471, MR1246666, MR1043716}. We refer the readers to~\cite{generalKleitman} more references and recent developments on Kleitman's theorem.

In this paper, we give a $q$-analog of Theorem \ref{Kleitman}. To state it, we need to generalize the Hamming graphs. Fix a prime power $q$ and let $V := \BF_q^n$ be the $n$-dimensional vector space over $\BF_q$. Denote by $\CV := \cb{\text{subspaces of } \BF_q^n}$ and $\CV(k) := \cb{A \in \CV: \dim A = k}$. The gaussian binomial coefficients record the cardinality of $\CV(k)$: $\abs{\CV(k)} = \qbinom{n}{k} := \frac{(q^n - 1)\cdots(q^{n - k + 1} - 1)}{(q^k - 1)\cdots(q - 1)}.$ The \emph{$n$-dimensional $q$-Hamming graph} has vertex set $\CV$, in which two subspaces $A$ and $B$ are joined by an edge if and only if $A \subseteq B$ and $\dim B - \dim A = 1$.
		
Our main result is an isodiametric inequality on $q$-Hamming graphs for large $n$.

\begin{theorem}\label{main_thm}
	
	Fix a prime power $q$. Let $n = d + 1$ or $n > 2d$ and let $G$ be the $n$-dimensional $q$-Hamming graph. Given $\CF \subseteq \CV$, if $\diam(\CF) \ls d$, then	
	\[\abs{\CF} \ls \left\{\begin{aligned}
		&\sum_{i = 0}^{t}\qbinom{n}{i}&\qquad&\text{if } d = 2t;\\
		&\sum_{i = 0}^{t}\qbinom{n}{i} + \qbinom{n - 1}{t}&\qquad&\text{if } d = 2t + 1.
	\end{aligned}\right.\]	
	
\end{theorem}

The bounds on the size of $\CF$ above are optimal. Indeed, set $\CF_1 = \bigcup_{k = 0}^{t}\CV(k)$ and $\CF_2 = \CF_1 \cup \cb{A \in \CV(t + 1): x \in A}$ for some fixed $x \in V - \cb{0}$. Then $\CF_1$ attains the maximum bound if $d = 2t$ and $\CF_2$ attains the maximum bound if $d = 2t + 1$.

An equivalent formulation of~\cref{main_thm} is an isodiametric inequality on a metric space $(\CV,\Delta)$, see~\eqref{eq:dist} and~\cref{lem:metric}. We believe that the isodiametric inequality holds for all $n>d$. It would be interesting to bridge the gap in~\cref{main_thm} when $d+2\le n\le 2d$.
    
\section{Preliminaries}

We need the following result on subspaces counting.

\begin{lemma}[\cite{MR3497070}]\label{counting} Fix $A \in \CV(k)$. Then $\abs{\cb{B \in \CV(\ell): \dim (A \cap B) = j}} = q^{(k - j)(\ell - j)}\qbinom{n - k}{\ell - j}\qbinom{k}{j}.$
	
\end{lemma}

Let $G$ be a finite bipartite graph with partite sets $X$ and $Y$ where $\abs{X} = \abs{Y}$. A \emph{perfect matching} is a set of disjoint edges which covers every vertex. For a subset $W \subseteq X$ (resp. $W \subseteq Y$), let $\Gamma(W)$ denote the set of all vertices in $Y$ (resp. $X$) that are adjacent to at least one vertex of $W$. Hall's theorem states that if for every subset $W \subseteq X$, $\abs{W} \ls \abs{\Gamma(W)}$, then $G$ has a perfect matching. We shall use the following folklore corollary. A \emph{regular graph} is a graph where each vertex has the same number of neighbors.

\begin{proposition}\label{coro_useful}
	
	Every finite regular bipartite graph has a perfect matching.
	
\end{proposition}


\subsection{A metric space $(\CV,\Delta)$}

For $A, B \in \CV$, define $\Delta: \CV \times \CV \rightarrow \BN$ as
\begin{equation}\label{eq:dist}
   \Delta(A, B) := \dim(A + B)-\dim (A\cap B) =\dim A + \dim B - 2\dim(A \cap B).    
\end{equation}

The symmetric positive-definite function $\Delta$ is a metric on $\CV$.

\begin{lemma}\label{triangle_inequality}
	
	The function $\Delta$ satisfies the triangle inequality.
	
\end{lemma}

\begin{proof}
	
	Recall the dimension formula from linear algebra, we have
	\begin{align*}
		\Delta(A, C)
		&\ls \dim A + \dim C - 2\dim (A \cap B \cap B \cap C) \\
		&= \dim A + \dim C - 2\dim(A \cap B) - 2\dim (B \cap C) + 2\dim((A \cap B) + (B \cap C)) \\
		&\ls \dim A + \dim C - 2\dim(A \cap B) - 2\dim(B \cap C) + 2\dim B\\
		&= \Delta(A, B) + \Delta(B, C).\qedhere
	\end{align*}	
	
\end{proof}

The following lemma shows that~\cref{main_thm} is equivalent to an isodiametric inequality on the metric space $(\CV,\Delta)$.

\begin{lemma}\label{lem:metric}
	
	Let $G$ be the $q$-Hamming graph. Then $\delta_G = \Delta$.
	
\end{lemma}

\begin{proof}
	
	Note that the union of an $A$-$(A \cap B)$ path and and $(A \cap B)$-$B$ path contains a path from $A$ to $B$. Thus, $\delta_G(A, B) \ls \delta_G(A,A\cap B)+\delta_G(A\cap B, B)=\Delta(A, B)$. 
    
    We shall prove the converse  $\Delta(A, B) \le \delta_G(A, B)$ by induction on $k=\delta_G(A,B)$. The cases $k \in \cb{0, 1}$ are trivial. Suppose now that $\delta_G(A, B) = k\ge 2$. Then there exists a vertex $C$ on the shortest path from $A$ to $B$, such that $\delta_G(A, C) = 1$ and $\delta_G(B, C) = k - 1$. By Lemma \ref{triangle_inequality} and by the induction hypothesis, we have $\Delta(A, B) \ls \Delta(A, C) + \Delta(B, C) = \delta_G(A, C) + \delta_G(B, C) = k. $
\end{proof}

For $\CF, \CG \subseteq \CV$, we say $(\CF, \CG)$ is \emph{cross}-$s$-\emph{intersecting in $V$} if $\dim(A \cap B) \gs s$ for any $A \in \CF$ and any $B \in \CG$. In particular, if $\CF = \CG$, then $\CF$ is $s$-\emph{intersecting in $V$}.

The following lemma follows immediately from~\eqref{eq:dist}.

\begin{lemma}\label{lem:equivform}
	
	Let $\CF \subseteq \CV$ with $\diam(\CF) = d$. Then $(\CF(i), \CF(j))$ is cross-$\ce{\pt{i + j - d}/2}$-intersecting in $V$ for any $i, j \in \supp(\CF)$.
	In particular, $\CF(k)$ is $(k - \fl{d/2})$-intersecting in $V$ for any $k \in \supp(\CF)$.
	
\end{lemma}

\begin{theorem}[\cite{MR867648}]\label{EKRVEC}
	Let $n \gs 2k - s$ and $\CF \subseteq \CV(k)$. If $\CF$ is $s$-intersecting in $V$, then
	\[\abs{\CF} \ls \max\cb{\qbinom{n - s}{k - s}, \qbinom{2k - s}{k - s}}.\]
	
\end{theorem}

\subsection{An isometry}

Recall that, given two metric spaces $(X_1, d_1)$ and $(X_2, d_2)$, a bijection $f: X_1 \rightarrow X_2$ is said to be an \emph{isometry} if $d_2(f(x), f(y)) = d_1(x, y)$ for any $x, y \in X_1$. Next we show that taking orthogonal complement is an isometry on $(\CV,\Delta)$. If $W$ is a subspace of $V$, then its \emph{orthogonal complement} $W^\perp$ is the subspace of $V$ consisting of all elements $v \in V$ such that $\an{v, w} = 0$ for all $w \in W$, where $\an{v, w} := \sum\limits_{i = 1}^{n}v_i \cdot w_i$. 

\begin{lemma}\label{lem:isometry}
	
	The bijection $f(U) := U^\perp$ is an isometry on $(\CV, \Delta)$.	
	
\end{lemma}

\begin{proof}
	
	Take arbitrary $U_1,U_2\in\CV$. Note that
	\begin{align*}
		\Delta(U_1^\perp, U_2^\perp)
		&= \dim(U_1^\perp + U_2^\perp) - \dim(U_1^\perp \cap U_2^\perp)= \dim((U_1 \cap U_2)^\perp) - \dim((U_1 + U_2)^\perp)\\
		&= n - \dim(U_1 \cap U_2) - n + \dim (U_1 + U_2)= \Delta(U_1, U_2).\qedhere
	\end{align*}	
	
\end{proof}

Given $\CF \subseteq \CV$, denote by $\diam(\CF) := \max_{A, B \in \CF}\Delta(A, B)$ its diameter and
	\[D(\CF) := \max\limits_{A, B \in \CF}\abs{\dim A - \dim B}.\]
	It follows from definition that $\abs{\dim A - \dim B} \ls \Delta(A, B)$, and so $D(\CF) \ls \diam(\CF)$. We write $\CF(i)\subseteq\CF$ for the set of subspaces of dimension $i$ in $\CF$. Let $\supp(\CF) := \cb{i \in \BN: \CF(i) \ne \varnothing}$ and $\CF^\perp := \cb{W^\perp: W \in \CF}$. If $\diam(\CF) = d$, define
	\[m_\CF := \min\cb{x \in \BN: \supp(\CF) \subseteq [x, x + d] \quad \text{ or } \quad \supp(\CF^\perp) \subseteq [x, x + d]}.\]

\begin{lemma}\label{lem:min}
	
	If $\diam(\CF) = d$, then $m_\CF \ls \pt{n - d}/2$.	
	
\end{lemma}

\begin{proof}
	
	It suffices to show that, if $y := \min\cb{x \in \BN: \supp(\CF) \subseteq [x, x + d]} > \pt{n - d}/2,$
	then we must have
	$\min\cb{x \in \BN: \supp\pt{\CF^\perp} \subseteq [x, x + d]} \ls (n - d)/2$.	
	But this follows from
	\[\supp(\CF) \subseteq [y, y + d] \quad \Rightarrow \quad \supp(\CF^\perp) \subseteq [n - d - y, n - y],\]
	where $n - d - y < n - d - (n - d)/2 = \pt{n - d}/2$.\qedhere
	
\end{proof}

By~\cref{lem:isometry,lem:min}, we may assume that
\begin{equation}\label{eq:min-dim}
	m_{\CF}=\min\cb{x \in \BN: \supp(\CF) \subseteq [x, x + d]} \ls \pt{n - d}/2.
\end{equation}


	\section{Proof of Theorem \ref{main_thm}}
	
Let us first consider the case $n = d + 1$. For each $k \ls n/2$, define a bipartite graph $G_k := \pt{\CV(k) \uplus \CV(n - k), \CE_k}$, where
\[\CE_k = \cb{\cb{A, B}: A \in \CV(k), B \in \CV(n - k) \text{ and } \dim(A \cap B) = 0}.\]
By Lemma \ref{counting}, we know that $G_k$ is ${q^{k(n - k)}}$-regular. By~\cref{coro_useful}, there is a perfect matching in $G_k$. Note that if $\cb{A, B}$ is an edge in $G_k$, then $\Delta(A, B) > d$. Thus for any $k \ls n/2$, we have
\[\abs{\CF(k)} + \abs{\CF(n - k)} \ls \qbinom{n}{k}.\]
Hence, if $d = 2t$, we have
\[\abs{\CF} = \sum_{i = 0}^{t}\pt{\abs{\CF(i)} + \abs{\CF(n - i)}} \ls \sum_{i = 0}^{t}\qbinom{n}{i}.\]
If $d = 2t + 1$, then $\CF(t + 1)$ must be $1$-intersecting in $V$. By Theorem \ref{EKRVEC}, we have $\abs{\CF(t + 1)} \ls \qbinom{n - 1}{t}$. So
\[\abs{\CF} = \sum_{i = 0}^{t}\pt{\abs{\CF(i)} + \abs{\CF(n - i)}} + \abs{\CF(t + 1)} \ls \sum_{i = 0}^{t}\qbinom{n}{i} + \qbinom{n - 1}{t}.\]

We may now assume $n>2d$. The cases $d \in \cb{0, 1}$	are trivial; assume then $d\ge 2$. Set $t := \fl{d/2} \gs 1$. We shall bound the size of each layer of $\CF$. By~\cref{lem:equivform}, for any $k \in \supp(\CF)$, $\CF(k)$ is a $(k - t)$-intersecting family. Hence by~\cref{EKRVEC}, we have
\begin{equation}\label{eq:slice}
 \abs{\CF(k)} \ls \max\cb{\qbinom{n - k + t}{t}, \qbinom{k + t}{t}}.    
\end{equation}

Let $m=m_{\CF}$. Then by~\eqref{eq:min-dim}, $m\le \frac{n-d}{2}$. Our strategy is to bound each of the slices $\CF(k)$. The cases when $d$ and $q$ are small require separate treatments.

\subsection{The case $d = 2$}

In this case, we have $t := \fl{d/2} = 1$ and $D(\CF)\le d=2$.
	
Suppose $D(\CF) = 0$, i.e., $\supp(\CF) = \cb{m}$. As $m\le \frac{n-d}{2}$, by~\eqref{eq:slice} we have
	\begin{equation}\label{eq:D0}
		\abs{\CF}=\abs{\CF(m)} \ls \max\cb{\qbinom{n - m + 1}{1}, \qbinom{m + 1}{1}} = \qbinom{n - m + 1}{1}\ls 1+\qbinom{n}{1},
	\end{equation}
	as desired.
    
Suppose then $D(\CF) = 1$, i.e., $\supp(\CF) = \cb{m, m + 1}$. By Lemma \ref{lem:equivform}, we have $(\CF(m), \CF(m + 1))$ is cross-$m$-intersecting in $V$. In other words, we have $X \subseteq Y$ for any $X \in \CF(m)$ and any $Y \in \CF(m + 1)$. By Lemma \ref{lem:equivform} again, we have $\CF(m)$ is $(m - 1)$-intersecting in $Y$ for any $Y \in \CF(m + 1)$.
	
Now, if $\abs{\CF(m + 1)} = 1$, i.e., $\CF(m + 1) = \cb{Y}$ for some $Y\in \CV(m+1)$, then by Theorem \ref{EKRVEC} with $V=Y$, we have
		\[\abs{\CF(m)} \ls \max\cb{\qbinom{2}{1}, \qbinom{m + 1}{1}} = \qbinom{m + 1}{1}.\]
	Thus, $\abs{\CF}\ls 1+\qbinom{m + 1}{1}\ls 1+\qbinom{n}{1}$.
		
	If instead $\abs{\CF(m + 1)} \gs 2$, i.e., $\CF(m + 1) = \cb{Y_1, Y_2, \ldots}$. Set $Z := Y_1 \cap Y_2$ and thus $\dim Z \ls m$. Hence $\abs{\CF(m)} \ls 1$. On the other hand, by~\eqref{eq:slice}, we have
		\[\abs{\CF(m + 1)} \ls \max\cb{\qbinom{n - m}{1}, \qbinom{m + 2}{1}} = \qbinom{n - m}{1}.\]
     Thus,  $\abs{\CF}\ls 1+\qbinom{n-m }{1}\ls 1+\qbinom{n}{1}$.  
			
Suppose $D(\CF) = 2$, i.e., $\supp(\CF) = \cb{m, m + 1, m + 2}$. By Lemma \ref{lem:equivform}, we have $(\CF(m), \CF(m + 2))$ is cross-$m$-intersecting in $V$ and $(\CF(m + 1), \CF(m + 2))$ is cross-$(m + 1)$-intersecting in $V$ respectively. In other words, we have $X \subseteq Y$ for any $X \in \CF(m) \cup \CF(m + 1)$ and any $Y \in \CF(m + 2)$. By Lemma \ref{lem:equivform} again, we have $\CF(m)$ is $(m - 1)$-intersecting in $Y$ for any $Y \in \CF(m + 2)$ and $\CF(m + 1)$ is $m$-intersecting in $Y$ for any $Y \in \CF(m + 2)$ respectively. 
	
If $\abs{\CF(m + 2)} = 1$, say $\CF(m + 2) = \cb{Y}$, then By Theorem \ref{EKRVEC}, we have
		\[\abs{\CF(m)} \ls \max\cb{\qbinom{3}{1}, \qbinom{m + 1}{1}},\]
		and
		\[\abs{\CF(m + 1)} \ls \max\cb{\qbinom{2}{1}, \qbinom{m + 2}{1}} = \qbinom{m + 2}{1}.\]

		If instead $\abs{\CF(m + 2)} \gs 2$, i.e., $\CF(m + 2) = \cb{Y_1, Y_2, \ldots}$. Set $Z := Y_1 \cap Y_2$ and thus $\dim Z \ls m + 1$. Then $\abs{\CF(m + 1)} \ls 1$ and by Theorem \ref{EKRVEC}, we have
		\[\abs{\CF(m)} \ls \max\cb{\qbinom{2}{1}, \qbinom{m + 1}{1}} = \qbinom{m + 1}{1}.\]
	By~\eqref{eq:slice}, we have
		\[\abs{\CF(m + 2)} \ls \max\cb{\qbinom{n - m - 1}{1}, \qbinom{m + 3}{1}}.\]

In each case, thanks to $\qbinom{n}{1} > \sum\limits_{i = 1}^{n - 1}\qbinom{i}{1}$, we have
\[\abs{\CF} = \sum_{k \in \supp(\CF)}\abs{\CF(k)} \ls 1 + \qbinom{n}{1}.\]
	
\subsection{The case $d = 3$}

In this case, we have $t := \fl{d/2} = 1$ and $D(\CF)\le d=3$.

If $D(\CF) = 0$, as in~\eqref{eq:D0}, we have
	\begin{equation*}
		\abs{\CF} \ls \max\cb{\qbinom{n - m + 1}{1}, \qbinom{m + 1}{1}} = \qbinom{n - m + 1}{1}\ls 1+\qbinom{n}{1}+\qbinom{n-1}{1}.
	\end{equation*}
	
If $D(\CF) = 1$, i.e.~$\supp(\CF) = \cb{m, m + 1}$. By~\eqref{eq:slice}, we have
	\[\abs{\CF(m)} \ls \max\cb{\qbinom{n - m + 1}{1}, \qbinom{m + 1}{1}} = \qbinom{n - m + 1}{1},\]
	and
	\[\abs{\CF(m + 1)} \ls \max\cb{\qbinom{n - m}{1}, \qbinom{m + 2}{1}} = \qbinom{n - m}{1}.\]
	Thus, $\abs{\CF}\ls\qbinom{n - m + 1}{1}+\qbinom{n - m }{1} \ls 1+\qbinom{n}{1}+\qbinom{n-1}{1}$.
	
Suppose $D(\CF) = 2$, i.e., $\supp(\CF) =\cb{m, m + 1, m + 2}$. Then by Lemma \ref{lem:equivform}, we have $(\CF(m), \CF(m + 2))$ is cross-$m$-intersecting in $V$ and $\CF(m + 1)$ is $m$-intersecting in $V$ respectively. In other words, we have $X \subseteq Y$ for any $X \in \CF(m)$ and any $Y \in \CF(m + 2)$. By Lemma \ref{lem:equivform} again, we have $\CF(m)$ is $(m - 1)$-intersecting in $Y$ for any $Y \in \CF(m + 2)$. By~\eqref{eq:slice},
		\[\abs{\CF(m + 1)} \ls \max\cb{\qbinom{n - m}{1}, \qbinom{m + 2}{1}} = \qbinom{n - m}{1}.\]
	
	If $\abs{\CF(m + 2)} = 1$, say $\CF(m + 2) = \cb{Y}$, then by Theorem \ref{EKRVEC}, we have
		\[\abs{\CF(m)} \ls \max\cb{\qbinom{3}{1}, \qbinom{m + 1}{1}}.\]
		
	If $\abs{\CF(m + 2)} \gs 2$, i.e., $\CF(m + 2) = \cb{Y_1, Y_2, \ldots}$. Set $Z := Y_1 \cap Y_2$ and thus $\dim Z \ls m + 1$. By Theorem \ref{EKRVEC} with $V_{\ref{EKRVEC}}=Z$, we have
		\[\abs{\CF(m)} \ls \max\cb{\qbinom{2}{1}, \qbinom{m + 1}{1}} = \qbinom{m + 1}{1},\]
		and by~\eqref{eq:slice},
		\[\abs{\CF(m + 2)} \ls \max\cb{\qbinom{n - m - 1}{1}, \qbinom{m + 3}{1}}.\]
In each case, thanks to $\qbinom{n}{1} > \sum\limits_{i = 1}^{n - 1}\qbinom{i}{1}$, we have
\[\abs{\CF} = \sum_{k \in \supp(\CF)}\abs{\CF(k)} \ls 1 + \qbinom{n}{1} + \qbinom{n - 1}{1}.\]
	
	Suppose $D(\CF) = 3$, i.e., $\supp(\CF) = \cb{m, m + 1, m + 2, m + 3}$. By Lemma \ref{lem:equivform}, we have $(\CF(m), \CF(m + 3))$ is cross-$m$-intersecting in $V$ and $(\CF(m + 1), \CF(m + 3))$ is cross-$(m + 1)$-intersecting in $V$ respectively. In other words, we have $X \subseteq Y$ for any $X \in \CF(m) \cup \CF(m + 1)$ and any $Y \in \CF(m + 3)$. By Lemma \ref{lem:equivform} again, we have $\CF(m)$ is $(m - 1)$-intersecting in $Y$ for any $Y \in \CF(m + 3)$ and $\CF(m + 1)$ is $m$-intersecting in $Y$ for any $Y \in \CF(m + 3)$ respectively. By~\eqref{eq:slice},
		\[\abs{\CF(m + 2)} \ls \max\cb{\qbinom{n - m - 1}{1}, \qbinom{m + 3}{1}}.\]
	
	If $\abs{\CF(m + 3)} = 1$, say $\CF(m + 3) = \cb{Y}$, then by Theorem \ref{EKRVEC} with $V_{\ref{EKRVEC}}=Y$, we have
		\[\abs{\CF(m)} \ls \max\cb{\qbinom{4}{1}, \qbinom{m + 1}{1}},\]
		and
		\[\abs{\CF(m + 1)} \ls \max\cb{\qbinom{3}{1}, \qbinom{m + 2}{1}}.\]
		
		If $\abs{\CF(m + 3)} \gs 2$, i.e., $\CF(m + 3) = \cb{Y_1, Y_2, \ldots}$. Set $Z := Y_1 \cap Y_2$, thus $\dim Z \ls m + 2$. By Theorem \ref{EKRVEC} with $V_{\ref{EKRVEC}}=Z$, we have
		\[\abs{\CF(m)} \ls \max\cb{\qbinom{3}{1}, \qbinom{m + 1}{1}},\]
		and
		\[\abs{\CF(m + 1)} \ls \max\cb{\qbinom{2}{1}, \qbinom{m + 2}{1}} = \qbinom{m + 2}{1},\]
		and by~\eqref{eq:slice},
		\[\abs{\CF(m + 3)} \ls \max\cb{\qbinom{n - m - 2}{1}, \qbinom{m + 4}{1}}.\]

In each case, thanks to $\qbinom{n}{1} > \sum\limits_{i = 1}^{n - 1}\qbinom{i}{1}$, we have
\[\abs{\CF} = \sum_{k \in \supp(\CF)}\abs{\CF(k)} \ls 1 + \qbinom{n}{1} + \qbinom{n - 1}{1}.\]
	
\subsection{The cases $d \gs 4$}

In these cases, we have $t := \fl{d/2} \gs 2$. Let us first assume that $m = m_\CF \gs t + 1$. Note that $\abs{\CF} = \sum_{k = m}^{\fl{n/2}}\abs{\CF(k)} + \sum_{k = \fl{n/2} + 1}^{m + d}\abs{\CF(k)}$. For the first sum, by~\eqref{eq:slice} and $m\ge t+1$ we have
	\begin{align}\label{eq:last}
		\qbinom{n}{t}^{-1} \cdot \sum_{k = m}^{\fl{n/2}}\abs{\CF(k)}
		&\ls \sum_{k = m}^{\fl{n/2}}\qbinom{n}{t}^{-1} \cdot \qbinom{n - k + t}{t}\ls \sum_{k = t + 1}^{\infty}\qbinom{n}{t}^{-1} \cdot \qbinom{n - k + t}{t}\nonumber\\
		&= \sum_{k = t + 1}^{\infty}\frac{(q^{n - k + t} - 1)\cdots(q^{n - k + 1} - 1)}{(q^n - 1)\cdots(q^{n - t + 1} - 1)}\nonumber\\
		&< \sum_{k = t + 1}^{\infty}\frac{q^{n - k + t} \cdots q^{n - k + 1}}{q^{n} \cdots q^{n - t + 1}}< \sum_{r = 1}^{\infty}q^{-tr}= \frac{1}{q^t - 1}.
	\end{align}

	For the second sum, recall that $m\le\frac{n-d}{2}$ and so
	\begin{align}\label{eq:last2}
		\qbinom{n}{t}^{-1} \cdot \sum_{k = \fl{n/2} + 1}^{m + d}\abs{\CF(k)}
		&\ls \sum_{k = \fl{n/2} + 1}^{\fl{(n - d)/2} + d} \qbinom{n}{t}^{-1} \cdot \qbinom{k + t}{t}= \sum_{k = \fl{n/2} + 1}^{\fl{(n + d)/2}} \frac{(q^{k + t} - 1)\cdots(q^{k + 1} - 1)}{(q^n - 1)\cdots(q^{n - t + 1} - 1)}\nonumber\\
		&\ls \sum_{k = \fl{n/2} + 1}^{\fl{(n + d)/2}} \frac{q^{k + t}\cdots q^{k + 1}}{q^n\cdots q^{n - t + 1}}< \sum_{r = 1}^{\infty} q^{-tr}= \frac{1}{q^t - 1},
	\end{align}
    where the penultimate inequality follows from $\fl{(n + d)/2}+t<n$. Thus, $\abs{\CF} < \frac{2}{q^t - 1}\cdot\qbinom{n}{t}$.

    We may then assume that $m = m_\CF \ls t$. Let $M = \max\supp(\CF)$ and further assume that
	\[M \gs \left\{\begin{aligned}
		&t + 1,&\qquad&\text{if } d = 2t;\\
		&t + 2,&\qquad&\text{if } d = 2t + 1.
	\end{aligned}\right.\]
   Then, we have
	\[\abs{\CF}\ls \sum_{i = 0}^{t - 1}\qbinom{n}{i} + \abs{\CF(t)} + \sum_{k = t + 1}^{M}\abs{\CF(k)}.\]

    Let us first bound the size of $\CF(t)$. To this end, fix $Y \in \CF(M)$ and note that for any $X \in \CF(t)$,
	\[\dim(X \cap Y) = \ce{\pt{\dim X + \dim Y - \Delta(X, Y)}/2} \gs \ce{\pt{t + M - d}/2} \gs 1.\]
	Thus by~\cref{counting}, we have
	
	\[\abs{\CF(t)} \ls \abs{\{X \in \CV(t): \dim(X \cap Y) \ne 0\}} = \qbinom{n}{t} - q^{Mt}\qbinom{n - M}{t}.\]
	
	Note that
	\begin{align*}
		\frac{q^{Mt}\qbinom{n - M}{t}}{\qbinom{n}{t}}
		&= \frac{(q^n - q^M)\cdots(q^{n - t + 1} - q^M)}{(q^n - 1)\cdots(q^{n - t + 1} - 1)}\gs \frac{(q^n - q^M)\cdots(q^{n - t + 1} - q^M)}{q^n\cdots q^{n - t + 1}}\\
		&= \prod_{i = 1}^{t}\pt{1 - \pt{q^{-1}}^{n - M - t + i}}\gs \prod_{i = 1}^{t}\pt{1 - \pt{q^{-1}}^{1 + i}}\gs \prod_{i = 1}^{t}\pt{1 - \pt{2^{-1}}^{1 + i}}\\
		&\gs \left\{\begin{aligned}
			&\pt{1 - 4^{-1}}\pt{1 - 8^{-1}} = \frac{21}{32},&\qquad&\text{if } t = 2;\\
			&2\prod_{i = 1}^{\infty}\pt{1 - \pt{2^{-1}}^{i}} \gs \frac{1}{2},&\qquad&\text{if } t \gs 3.
		\end{aligned}\right.
	\end{align*}
To see the last inequality $y := \prod\limits_{i = 1}^{\infty}\pt{1 - 2^{-i}} \gs \frac{1}{4}$, set $f(x) = \prod\limits_{i = 1}^{\infty}\pt{1 - x^i}^{-1}$ where $x < 2/3$, it is the generating function of the \emph{partition function} $p(k)$, which is the number of distinct ways of representing $k$ as a sum of positive integers. Note that $p(k) \ls e^{\pi\sqrt{2k/3}}$ (see\cite{MR2471621}) and $e^{\pi\sqrt{2k/3}} \ls (3/2)^k$ for $k \gs 41$; it can be verified (OEIS: A000041) that $p(k) \ls (3/2)^k$ for $1 \ls k \ls 40$. Thus, $f(x) = \sum\limits_{k = 0}^{\infty}p(k)x^k \ls \sum\limits_{k = 0}^{\infty}(3x/2)^k = {2}/\pt{2 - 3x}$, and $y=1/f(\frac{1}{2}) \gs 1/(\frac{2}{2-3/2}) = 1/4$. Therefore, $\abs{\CF(t)}\ls c_1\qbinom{n}{t}$, where $c_1=11/32$ if $t=2$ and $c_1=1/2$ if $t\ge 3$.

To bound $\sum\limits_{k = t + 1}^{M}\abs{\CF(k)}$, if $\max\cb{q, t} > 2$, then by~\eqref{eq:last} and~\eqref{eq:last2} we have $\sum\limits_{k = t + 1}^{M}\abs{\CF(k)} \ls \frac{2}{q^t - 1}\qbinom{n}{t}\ls \frac{2}{7}\qbinom{n}{t}$. If $q=t=2$, then note that there are at most $5$ summands if $t = 2$ since $M \ls t + d$. In this case, based on the calculation in~\eqref{eq:last} and~\eqref{eq:last2}, $\qbinom{n}{t}^{-1}\cdot\sum\limits_{k = t + 1}^{M}\abs{\CF(k)} \ls \sum\limits_{a \in A}a$, where $A$ is a $5$-term-subset of the multiset $\cb{1/4, 1/4, 1/16, 1/16, 1/64, 1/64, 1/256, 1/256, \ldots}$, then $\sum\limits_{k = t + 1}^{M}\abs{\CF(k)} \ls \frac{41}{64}\qbinom{n}{t}$. 
	
	Put the upper bounds together, we have
    \begin{align*}
		\abs{\CF}
		&\ls \sum_{i = 0}^{t - 1}\qbinom{n}{i} + \abs{\CF(t)} + \sum_{k = t + 1}^{M}\abs{\CF(k)}\\
		&\ls \sum_{i = 0}^{t - 1}\qbinom{n}{i} + c_1\qbinom{n}{t} + c_2\qbinom{n}{t}\\
		&< \sum_{i = 0}^{t}\qbinom{n}{i}.
	\end{align*}
	where $(c_1, c_2) = \pt{11/32, 41/64}$ if $q = t = 2$, and $\pt{1/2, 2/7}$ if $\max\cb{q, t} > 2$.

	We are left with the case when 
	\[M \ls \left\{\begin{aligned}
		&t,&\qquad&\text{if } d = 2t;\\
		&t + 1,&\qquad&\text{if } d = 2t + 1.
	\end{aligned}\right.\]
	When $d=2t$, $\abs{\CF}=\sum_{i\ls M}\abs{\CF(i)}\le \sum_{i = 0}^{t}\qbinom{n}{i}$. When $d=2t+1$, note that $\CF(t + 1)$ is $1$-intersecting in $V$ if $d = 2t + 1$, so $\abs{\CF(t + 1)} \ls \qbinom{n - 1}{t}$ by~\eqref{eq:slice}. Therefore $\abs{\CF}=\sum_{i\ls M}\abs{\CF(i)}\le \sum_{i = 0}^{t}\qbinom{n}{i}+\qbinom{n - 1}{t}$. 
    
    This completes the proof.
	
	\nocite{*}
	
	\bibliographystyle{abbrv}
	
	\bibliography{main}

\end{document}